\documentclass[11pt]{amsart}
\usepackage{amssymb,amsmath,epsfig,mathrsfs, enumerate}
\usepackage{graphicx}
\usepackage[normalem]{ulem}
\usepackage{fancyhdr}
\usepackage[margin=3cm]{geometry}
\usepackage{hyperref}
\hypersetup{
 colorlinks   = true,
 urlcolor     = blue,
 linkcolor    = blue,
 citecolor   = red ,
 bookmarksopen=true
}

\theoremstyle{plain}
\newtheorem{thrm}{Theorem}[section]
\newtheorem{lemma}[thrm]{Lemma}

\begin{document}
\newcommand{\sn}{\mathbb{S}^{n-1}}
\newcommand{\SL}{\mathcal L^{1,p}( D)}
\newcommand{\Lp}{L^p( Dega)}
\newcommand{\CO}{C^\infty_0( \Omega)}
\newcommand{\Rn}{\mathbb R^n}
\newcommand{\Rm}{\mathbb R^m}
\newcommand{\R}{\mathbb R}
\newcommand{\Om}{\Omega}
\newcommand{\Hn}{\mathbb H^n}
\newcommand{\aB}{\alpha B}
\newcommand{\eps}{\ve}
\newcommand{\BVX}{BV_X(\Omega)}
\newcommand{\p}{\partial}
\newcommand{\IO}{\int_\Omega}
\newcommand{\bG}{\boldsymbol{G}}
\newcommand{\bg}{\mathfrak g}
\newcommand{\bz}{\mathfrak z}
\newcommand{\bv}{\mathfrak v}
\newcommand{\Bux}{\mbox{Box}}
\newcommand{\e}{\ve}
\newcommand{\X}{\mathcal X}
\newcommand{\Y}{\mathcal Y}
\newcommand{\W}{\mathcal W}
\newcommand{\la}{\lambda}
\newcommand{\vf}{\varphi}
\newcommand{\rhh}{|\nabla_H \rho|}
\newcommand{\Ba}{\mathcal{B}_\beta}
\newcommand{\Za}{Z_\beta}
\newcommand{\ra}{\rho_\beta}
\newcommand{\na}{\nabla_\beta}
\newcommand{\vt}{\vartheta}

\numberwithin{equation}{section}

\newcommand{\RN} {\mathbb{R}^N}
\newcommand{\Sob}{S^{1,p}(\Omega)}
\newcommand{\Dxk}{\frac{\partial}{\partial x_k}}
\newcommand{\Co}{C^\infty_0(\Omega)}
\newcommand{\Je}{J_\ve}
\newcommand{\beq}{\begin{equation}}
\newcommand{\bea}[1]{\begin{array}{#1} }
\newcommand{\eeq}{ \end{equation}}
\newcommand{\ea}{ \end{array}}
\newcommand{\eh}{\ve h}
\newcommand{\Dxi}{\frac{\partial}{\partial x_{i}}}
\newcommand{\Dyi}{\frac{\partial}{\partial y_{i}}}
\newcommand{\Dt}{\frac{\partial}{\partial t}}
\newcommand{\aBa}{(\alpha+1)B}
\newcommand{\GF}{\psi^{1+\frac{1}{2\alpha}}}
\newcommand{\GS}{\psi^{\frac12}}
\newcommand{\HFF}{\frac{\psi}{\rho}}
\newcommand{\HSS}{\frac{\psi}{\rho}}
\newcommand{\HFS}{\rho\psi^{\frac12-\frac{1}{2\alpha}}}
\newcommand{\HSF}{\frac{\psi^{\frac32+\frac{1}{2\alpha}}}{\rho}}
\newcommand{\AF}{\rho}
\newcommand{\AR}{\rho{\psi}^{\frac{1}{2}+\frac{1}{2\alpha}}}
\newcommand{\PF}{\alpha\frac{\psi}{|x|}}
\newcommand{\PS}{\alpha\frac{\psi}{\rho}}
\newcommand{\ds}{\displaystyle}
\newcommand{\Zt}{{\mathcal Z}^{t}}
\newcommand{\XPSI}{2\alpha\psi \begin{pmatrix} \frac{x}{|x|^2}\\ 0 \end{pmatrix} - 2\alpha\frac{{\psi}^2}{\rho^2}\begin{pmatrix} x \\ (\alpha +1)|x|^{-\alpha}y \end{pmatrix}}
\newcommand{\Z}{ \begin{pmatrix} x \\ (\alpha + 1)|x|^{-\alpha}y \end{pmatrix} }
\newcommand{\ZZ}{ \begin{pmatrix} xx^{t} & (\alpha + 1)|x|^{-\alpha}x y^{t}\\
     (\alpha + 1)|x|^{-\alpha}x^{t} y &   (\alpha + 1)^2  |x|^{-2\alpha}yy^{t}\end{pmatrix}}
\newcommand{\norm}[1]{\lVert#1 \rVert}
\newcommand{\ve}{\varepsilon}

\title[A strong unique continuation property, etc.]{A strong unique continuation property for the heat  operator with Hardy type potential}

\author{Agnid Banerjee}
\address{Tata Institute of Fundamental Research\\
Centre For Applicable Mathematics \\ Bangalore-560065, India}\email[Agnid Banerjee]{agnid@tifrbng.res.in}

\author{Nicola Garofalo}
\address{Dipartimento di Ingegneria Civile, Edile e Ambientale (DICEA) \\ Universit\`a di Padova\\ 35131 Padova, ITALY}
\email[Nicola Garofalo]{nicola.garofalo@unipd.it}

\author{Ramesh Manna}
\address{Department of Mathematics, Indian Institute of Science, 560 012 Bangalore, India}
\email{rameshmanna@iisc.ac.in}

\thanks{First author is supported in part by SERB Matrix grant MTR/2018/000267 and by Department of Atomic Energy,  Government of India, under
project no.  12-R \& D-TFR-5.01-0520.}

\thanks{Second author is supported in part by the Progetto SID (Investimento Strategico di Dipartimento): ``Non-local operators in geometry and in free boundary problems, and their connection with the applied sciences", University of Padova, 2017, and by the Progetto SID: ``Non-local Sobolev and isoperimetric inequalities", University of Padova, 2019.}
\thanks{Third author is  supported by  C.V. Raman PDF, R(IA)CVR-PDF/2020/224}

%
%
%

\maketitle
\begin{abstract}
In this note we prove the strong unique continuation property at the origin for the solutions of the parabolic differential inequality 
\[
|\Delta u - u_t| \leq \frac{M}{|x|^2} |u|,
\]
with the critical inverse square potential. Our main result sharpens a previous one of Vessella concerned with the subcritical case. 

\end{abstract}

\section{Introduction}
The unique continuation property (ucp) for second order elliptic and parabolic equations represents one of the most fundamental aspects of pde's with a long history  and several important ramifications. In this paper we prove the strong unique continuation property (sucp) for solutions to the parabolic differential inequality 
\begin{equation}\label{par}
|\Delta u - u_t| \leq \frac{M}{|x|^2}\ |u|,
\end{equation}
where $M > 0$ is arbitrary. In \cite{V} (see also \cite{EV}) Vessella proved a general sucp result for sub-critical parabolic equations of the type 
\begin{equation}\label{v}
|\operatorname{div} (A(x,t) \nabla u) - u_t| \leq \frac{M}{|x|^{2-\delta}} |u|,\ \ \ \ \delta>0,
\end{equation}
under Lipschitz regularity assumptions on the principal part $A(x,t)$. This provided a parabolic counterpart to the previous work of H\"ormander in \cite{Ho}. Comparing \eqref{par} with \eqref{v}, it is obvious that our result sharpens Vessella's theorem, when the latter is specialised to the heat equation. 

As it is well-known, the inverse-square potential $V(x) = \frac{M}{|x|^2}$ represents a critical scaling threshold in quantum mechanics \cite{BG}, and it is equally known that its singularity is the limiting case for the sucp for the differential inequality $|\Delta u| \le \frac{M}{|x|^m} |u|$, see the counterexample in \cite{GL}. Such potential fails to be in $L^{n/2}_{loc}$, and in general does not have small $L^{n/2,\infty}$ seminorm, thus in the context of the Laplacian the sucp cannot be treated by the celebrated result of Jerison and Kenig in \cite{JK} or the subsequent improvement by Stein in the appendix to the same paper. 
We recall that, in the time-independent case of the Laplacian, the sucp for the unrestricted inverse square potential was proved by Pan in \cite{Pa}. One should also see Regbaoui \cite{RR} for further generalisations to variable coefficient equations and Lin, Nakamura \& Wang \cite{LNW} for quantitative results. 

The main new  ingredient in this note is the following   improved Carleman estimate for  the heat operator $\Delta -\partial_t$ in a space-time cylinder that is tailor-made for the differential inequality  \eqref{par}. Such result replaces the corresponding sub-critical estimate in \cite[Theorem 13]{V} (see  also \cite[Theorem 2]{EV}). Similarly to the time-independent case in \cite{Pa} and \cite{RR}, our proof of Theorem \ref{thm2} also exploits the spectral gap on $\sn$. In addition, it relies in an essential way on a delicate a priori estimate which we prove in Lemma \ref{L:del} below, and which we feel is of independent interest. We emphasise for the unfamiliar reader that, although related sub-critical estimates appear in the works \cite{EV}, \cite{V}, the weight in our Carleman estimate \eqref{est1} is different from that in such works, and a new result was required.

\begin{thrm} \label{thm2}
 Let $R<1$ and let  $u \in C_0^{\infty}((B_R \setminus \{0\}) \times (0,T))$. There exist universal constants $\alpha(n)>>1$, and  $\ve(n)  \in (0, 1)$,   such that for all $\alpha>\alpha(n)$ of the form $\alpha = k + \frac{n+1}2$, with $k\in \mathbb N$, and every $0<\ve<\ve(n)$, one has 
\begin{align}\label{est1}
&\alpha \int_{B_R \times (0,T)} |x|^{-2\alpha-4} e^{2\alpha |x|^{\ve}} \, u^2  dxdt  +\alpha^3 \int_{B_R \times (0, T)} |x|^{-2\alpha- 4+\ve} e^{2\alpha |x|^{\ve}}  u^2 \\  
&\leq  C \int_{B_R \times (0, T)} |x|^{-2\alpha} e^{2\alpha |x|^{\ve}} (\Delta u - \partial_t u)^2   dx dt,\notag
\end{align}
where $C= C(\ve, n)>0$.
\end{thrm}

With Theorem \ref{thm2} in hands, we establish the following strong unique continuation result. In the sequel, parabolic vanishing to infinite order means that as $r \to 0$ one has for all $k>0$,
\begin{equation}\label{vp}
\int_{B_r \times (0, T)} u^2 = O(r^k).
\end{equation}

\begin{thrm}\label{thm1}
Suppose that for some $M, R, T>0$ the function $u \in H^{2,1}_{loc}$ be a solution in $B_R \times (0, T)$ to the  differential inequality \eqref{par}. If $u$ parabolically vanishes to infinite order, then $u \equiv 0$ in $B_R \times (0, T)$. 
\end{thrm}

For the reader's comprehension we remark that the first integral in the left-hand side of \eqref{est1} represents the critical term  which, in the proof of Theorem \ref{thm1}, allows to absorb the term with the inverse square potential in the differential inequality \eqref{par}. The   second integral, instead,  can be thought of as a sub-critical term. As the proof of Theorem \ref{thm1} will show the presence of the coefficient $\alpha^3$ in front of such term plays a crucial role  in the proof of Theorem \ref{thm1}, similarly to what happens in \cite{V}.

The plan of the paper is as follows. In Section \ref{s:main} we prove Theorem \ref{thm2}. The second part of the section is devoted to proving the crucial Lemma  \ref{L:del} which is needed for the completion of the proof of Theorem \ref{thm2}. In Section \ref{S:2} we establish Theorem \ref{thm1}. One word of caution for the reader. It is generally accepted among experts that, once a proper Carleman estimate is available, the  ucp, or the sucp follow from a standard application of the former. While this is generally true, in the time-dependent setting of the present note deducing Theorem 
 \ref{thm1} from Theorem \ref{thm2} requires a delicate adaptation of the analogous treatment of the subcritical case in \cite{V}. For this reason, we have not followed the tradition of skipping details, but we have carefully presented them in the proof of Theorem \ref{thm1}. 
 
\medskip

\noindent \textbf{Acknowledgment:} This work began during a visit of the second named author to the Tata Institute of Fundamental Research, Centre For Applicable Mathematics, in Bangalore, during February 2020. He gratefully acknowledges the gracious hospitality of the Centre and the warm work environment.

\section{Proof of Theorem \ref{thm2}}\label{s:main}

We begin by introducing the relevant notation.
Given $r>0$ we  denote by $B_r(x_0)$ the Euclidean ball centred at  $x_0 \in \R^n$ with radius $r$. When $x_0=0$, we will use the simpler notation $B_r$.   A generic point in space time $\Rn \times (0, \infty)$ will be denoted by $(x,t)$. For notational convenience, $\nabla f$ and  $\operatorname{div}\ f$ will respectively refer to the quantities  $\nabla_x f$ and $ \operatorname{div}_x f$ of a given function $f$.   The partial derivative in $t$ will be denoted by $\p_t f$ and also by $f_t$. We indicate with $C_0^{\infty}(\Omega)$ the set of compactly supported smooth functions in the region $\Omega$  in space-time. By $H^{2,1}_{loc}(\Omega)$ we refer to the parabolic Sobolev class of functions $f\in L^2_{loc}(\Om)$ for which the weak derivatives $\nabla f, \nabla^2 f $ and $\partial_t f$ belong to $L^{2}_{loc}(\Om)$. For a point $x\in \Rn\setminus\{0\}$, we will routinely adopt the notation $r = r(x) = |x|$ and $\omega = \frac{x}{r}\in \sn$, so that $x = r \omega$. The radial derivative of a function $v$ is 
$v_r = <\nabla v,\frac{x}{|x|}>$.

The following simple observations will be repeatedly used in what follows. Let $\gamma\in \R$, then in $\Rn\setminus\{0\}$ we have
\begin{equation}\label{div}
\operatorname{div}(r^{-\gamma} x) = (n-\gamma) r^{-\gamma}.
\end{equation}
In particular, \eqref{div} gives
\begin{lemma}\label{L:radial}
Let $f\in C^\infty_0(\Rn\setminus\{0\})$, $g\in C^\infty(\Rn\setminus\{0\})$, then
\[
\int_{\Rn} f_r g dx = - \int_{\Rn} f g_r dx - (n-1) \int_{\Rn} r^{-1} f g dx.
\]
\end{lemma}

\begin{proof}
It suffices to observe that \eqref{div} gives
\[
\operatorname{div}(fg r^{-1} x) = f g_r + g f_r +(n-1) r^{-1} fg.
\]
Integrating this identity we reach the desired conclusion.

\end{proof}

\begin{proof}[Proof of Theorem \ref{thm2}] 
In all subsequent integrals, for given $R\in (0,1)$, $T>0$, the domain of integration will be the parabolic cylinder $B_R \times (0,T)$ (or, for that matter, the whole of $\Rn\times \R$, in view of the support property of the integrands), but this will not be explicitly indicated. Nor, we will explicitly write the measure $dx dt$ in any of the integrals involved. Let $u \in C_0^{\infty}((B_R \setminus \{0\}) \times (0,T))$.  We set $v=r^{-\beta} e^{\alpha r^\ve}  u$, where $\beta$  is to be carefully chosen subsequently. Clearly,  $u= r^{\beta} e^{-\alpha r^\ve} v$.  A standard calculation shows
 \[
 \Delta(r^{\beta} e^{-\alpha r^\ve})=\left(\alpha^2 \ve^2 r^{\beta+2\ve-2} + \beta(\beta+n-2)r^{\beta-2}- \alpha\ve\left((2\beta+\ve+ n-2)\right) r^{\beta+\ve-2}\right) e^{-\alpha r^{\ve}}.
 \]
We thus have
 \begin{align*}
\Delta u & =r^{\beta} e^{-\alpha r^\ve}  \Delta v+\left(\alpha^2 \ve^2 r^{\beta+2\ve-2} + \beta(\beta+n-2)r^{\beta-2}- \alpha\ve\left((2\beta+\ve+ n-2)\right) r^{\beta+\ve-2}\right) e^{-\alpha r^{\ve}} v
\\
& + \left(2\beta r^{\beta-2} -2 \ve \alpha r^{\beta+\ve-2}\right) e^{-\alpha r^{\ve}} <x, \nabla v>.
\end{align*}
Since $\Delta v(x,t) =v_{rr}(r\omega,t)+\frac{n-1}{r} v_r(r\omega,t)+\frac{1}{r^2} \Delta_{\sn} v(r\omega,t),$
where $\omega\in \sn$ and $\Delta_{\sn}$ denotes the Laplacian on $\mathbb{S}^{n-1}$,  we obtain
\begin{align}\label{comp1}
& \Delta u -\partial_tu = r^\beta e^{-\alpha r^{\ve}} \bigg[\left(\alpha^2 \ve^2 r^{2\ve-2} + \beta(\beta+n-2)r^{-2}- \alpha\ve\left(2\beta+\ve+ n-2\right) r^{\ve-2}\right) v
\\
& + \left((2\beta+n-1) r^{-1} -2 \alpha \ve r^{\ve-1}\right) v_r + r^{-2} \Delta_{\sn} v+ v_{rr} - v_t\bigg].\notag
\end{align}
We now apply the numerical inequality $(a+b)^2 \geq a^2 + 2ab$, with 
\[
a= r^{\beta -2} e^{-\alpha r^\ve} \left( \beta(\beta+n-2) v + \Delta_{\sn} v + (2\beta+ n-1)r v_r  \right),
\] 
and 
\[
b = r^\beta e^{-\alpha r^\ve}  \left(\alpha^2 \ve^2 r^{2\ve-2} v - \alpha\ve(2\beta+\ve+ n-2) r^{\ve-2} v
-2 \alpha \ve r^{\ve-1} v_r + v_{rr} - v_t\right),
\]
obtaining 
\begin{align}\label{rt12nd}
&\int r^{-2\alpha}  e^{2\alpha r^\ve}(\Delta u - \partial_t u)^2  \geq \int r^{-2\alpha+2\beta-4} (\beta(\beta+n-2) v + \Delta_{\sn}v )^2
\\
&  + (2\beta+n-1)^2 \int r^{-2\alpha + 2\beta -2} v_r^2 + 2(2\beta+n-1)\int r^{-2\alpha+2\beta-3} v_r \Delta_{\sn}v
\notag
\\
& +2\beta(\beta+n-2)(2\beta+n-1)\int r^{-2\alpha+2\beta-3} v v_r
\notag
\\
& +2\beta(\beta+n-2)\int r^{-2\alpha+2\beta-2}v v_{rr} \notag
 +2\int r^{-2\alpha+2\beta-2}v_{rr}\Delta_{\sn}v
 \notag\\
 & + 2(2\beta+n-1) \int r^{-2\alpha + 2\beta -1} v_r v_{rr} -2\beta(\beta+n-2)\int r^{-2\alpha+2\beta-2} v v_t
 \notag
 \\
 & -2\int r^{-2\alpha+2\beta-2} v_t \Delta_{\sn}v  -2 (2\beta+n-1) \int r^{-2\alpha+2\beta-1} v_t v_r 
 \notag\\
& - 2\alpha \ve (2\beta+n-1) (2 \beta+ \ve + n-2) \int r^{-2\alpha + 2\beta+\ve -3} vv_r 
+ 2 \alpha^2 \ve^2 (2\beta+n-1)  \int r^{-2\alpha+2\beta+2\ve -3} vv_r 
\notag\\
& - 4 \alpha \ve( 2\beta+n-1)  \int r^{-2 \alpha +2\beta + \ve -2} v_r^2 - 4  \alpha \ve \beta(\beta+n-2) \int r^{-2\alpha+2\beta +\ve -3} vv_r 
\notag\\
&    -  2 \alpha \ve \beta(\beta+n-2)  (2 \beta+ \ve+n-2)  \int r^{-2\alpha + 2\beta +\ve -4} v^2 
 +2 \alpha^2 \ve^2 \beta(\beta+n-2)   \int r^{-2\alpha +2\beta +2\ve -4} v^2
 \notag\\
 &  - 4  \alpha  \ve \int r^{-2\alpha +2\beta+ \ve-3 } v_r \Delta_{\sn} v   -  2 \alpha\ve\left(2\beta+\ve+ n-2\right) \int r^{-2\alpha + 2\beta + \ve -4} v \Delta_{\sn} v 
\notag\\
& +2 \alpha^2\ve^2  \int r^{-2\alpha +2\beta +2\ve -4} v \Delta_{\sn} v.
\notag
\end{align}
We now handle each integral in the right-hand side of \eqref{rt12nd} separately. Our first objective is to select $\beta$ in such a way that the integral $\int r^{-2\alpha+2\beta-3} v v_r
$ vanishes. We note that such integral multiplies the cubic factor $2\beta(\beta+n-2)(2\beta+n-1)$ in the forth term in the right-hand side of \eqref{rt12nd}. To accomplish this we observe that Lemma \ref{L:radial} gives 
\begin{align}\label{1}
& 2\int r^{-2\alpha+2\beta-3} v  v_r = \int r^{-2\alpha+2\beta-3} (v^2)_r
\\
& = - (-2\alpha+2\beta-3) \int r^{-2\alpha+2\beta-4} v^2 - (n-1) \int r^{-2\alpha+2\beta-4} v^2 = 0,
\notag
\end{align}
provided that we choose 
\begin{eqnarray}\label{choice}
2\beta -2\alpha-4=-n\ \Longleftrightarrow\ \beta = \alpha + 2 - \frac{n}2.
\end{eqnarray}
We now substitute the value of $\beta$ given by \eqref{choice} in the remaining integrals in the right-hand side of \eqref{rt12nd} obtaining the following conclusions. First, we  have
\[
\int r^{-2\alpha+2\beta-2} v v_t = \int r^{-n+2} v v_t = \int \partial_t (r^{-n+2}v^2)=0.
\]
Next, using polar coordinates and Stokes' theorem on $\sn$, we find
\begin{align*}
& 2 \int r^{-2\alpha+2\beta-2} v_t \Delta_{\sn}v = 2 \int r^{-n+2} v_t \Delta_{\sn}v = 2 \int_0^T \int_0^{\infty} r \int_{\mathbb{S}^{n-1}} v_t \Delta_{\sn}v d \omega dr dt\\
&= - 2 \int_0^T \int_0^{\infty} r \int_{\mathbb{S}^{n-1}} <\nabla_{\sn}v, \nabla_{\sn}v_t>dr d\omega dt=- \int_{0}^{\infty} r \int_0^{T}  \partial_t (\int_{\mathbb{S}^{n-1}} |\nabla_{\sn}v|^2 d\omega) dt dr = 0.
\end{align*} 
Similarly, we have
\begin{align} \label{zero}
& 2 \int r^{-2\alpha+2\beta-3} v_r \Delta_{\sn}v = 2 \int r^{-n+1} v_r \Delta_{\sn}v 
\\
& = 2 \int_0^T \int_0^{\infty}\int_{\mathbb{S}^{n-1}} v_r \Delta_{\sn} v d \omega dr dt = - 2 \int_0^T\int_0^{\infty}\int_{\mathbb{S}^{n-1}} \langle \nabla_{\sn} v_r, \nabla_{\sn} v \rangle d \omega dr dt
\nonumber
\\
& = - \int_0^T \int_0^{\infty}\frac{d}{dr}\int_{\mathbb{S}^{n-1}} |\nabla_{\sn} v|^2 d \omega dr dt=0. \nonumber
\end{align}
Now, an integration by parts similar to \eqref{zero} gives
\begin{equation}\label{sp1}
- 4  \alpha \ve \int r^{-2\alpha + 2\beta + \ve - 3} v_r \Delta_{\sn} v  = -2  \alpha  \ve^2 \int r^{-n+\ve} |\nabla_{\sn} v|^2.
\end{equation}
On the other hand, applying again the divergence theorem on $\sn$, we find
\begin{align}\label{sp2}
& -  2 \alpha\ve\left(2\beta+\ve+ n-2\right)\int r^{-2\alpha+2\beta + \ve - 4} v \Delta_{\sn} v
 = 2 \alpha\ve\left(2\alpha+\ve+ 2\right)  \int r^{-n+\ve}  |\nabla_{\sn} v|^2,
\end{align}
and
\begin{equation}\label{sp22}
2\alpha^2 \ve^2  \int r^{-2\alpha+2\beta + 2\ve - 4} v \Delta_{\sn} v=- 2\alpha^2 \ve^2    \int r^{-n+2\ve}  |\nabla_{\sn} v|^2.
\end{equation}
Keeping in mind that on the domain of integration we have $0<r\le R<1$, from \eqref{sp1}, \eqref{sp2} and \eqref{sp22} we deduce that if $\ve$ is sufficiently small, for instance, $0<\ve\le\frac{3}{20}$ would do, we can guarantee that
\begin{align}\label{sp4}
&- 4 \ve \alpha  \int r^{-2\alpha +2\beta+ \ve-3 } v_r \Delta_{\sn} v -  2 \alpha\ve\left(2\beta+\ve+ n-2\right) \int r^{-2\alpha + 2\beta + \ve -4} v \Delta_{\sn} v
\\
&  +2 \alpha^2\ve^2  \int r^{-2\alpha +2\beta +2\ve -4} v \Delta_{\sn} v\geq \frac{37}{10} \alpha^2 \ve  \int r^{-n+\ve} |\nabla_{\sn} v|^2.
\notag
\end{align}
We also claim that
\begin{equation}\label{kt2}
\int r^{-2\alpha+2\beta-2} v v_{rr} =  \int r^{-n+2} v v_{rr} =- \int r^{-n+2} v_{r}^2. 
\end{equation}
To see \eqref{kt2} we apply Lemma \ref{L:radial} with $g = r^{-n+2} v$ and $f = v_r$, obtaining
\[
\int r^{-2\alpha+2\beta-2} v v_{rr} =  \int r^{-n+2} v v_{rr} =- \int r^{-n+2} v_{r}^2 + (n-2) \int r^{-n+1} v v_r. 
\]
Since the last term vanishes in view of \eqref{1}, we conclude that \eqref{kt2} holds. Next, again by Lemma \ref{L:radial} we have 
\begin{equation}\label{vr}
\int r^{-2\alpha + 2\beta -1} v_r v_{rr}= \frac{1}{2} \int r^{-n+3} (v_r^2)_r  = - \int r^{-n+2} v_r^2.
\end{equation}
Yet another application of Lemma \ref{L:radial} gives
\begin{align} \label{3.9}
& 2\int r^{-2\alpha+2\beta-2}v_{rr}\Delta_{\sn}v = 2\int r^{-n+2} v_{rr} \Delta_{\sn} v 
\\
&= -2 \int r^{-n+2} v_r \Delta_{\sn} v_r - 2 (n-1)\int r^{-n+1} v_r \Delta_{\sn} v    
\nonumber\\
&=2\int_{0}^T \int_0^{\infty} r\int_{\mathbb{S}^{n-1}} |\nabla_{\sn} v_r|^2 d\omega dr dt \geq 0.
\nonumber
\end{align}
Note that in the third equality above we have used that
$\int r^{-n+1} \Delta_{\sn} v v_r =0$,
a fact which was earlier established in \eqref{zero}. 

A further application of Lemma \ref{L:radial} gives
\begin{equation*}\label{vvr}
2 \int r^{-2\alpha + 2\beta+\ve -3} vv_r = - \ve \int r^{-n+\ve} v^2.
\end{equation*}
Using this observation along with \eqref{choice} we obtain
\begin{align*}
&- 2 \alpha \ve (2\beta+n-1)  (2 \beta+ \ve + n-2) \int r^{-2\alpha + 2\beta+\ve -3} vv_r  + 2 \alpha^2 \ve^2 (2\beta+n-1)  \int r^{-2\alpha+2\beta+2\ve -3} vv_r \\
&- 4  \alpha \ve \beta(\beta+n-2) \int r^{-2\alpha+2\beta +\ve -3} vv_r 
\\
& = \alpha \ve^2 \bigg[ (2\alpha+\ve)  (2 \alpha+ \ve + 2) - \alpha \ve (2\alpha + 3) + (2\alpha+4 - n)(2\alpha+n)\bigg] \int r^{-n+\ve} v^2.
\end{align*}
On the other hand, again by \eqref{choice} we have
\begin{align*}
& -  2 \alpha \ve \beta(\beta+n-2) (2\beta+\ve+n-2)  \int r^{-2\alpha + 2\beta +\ve -4} v^2 
+ 2 \alpha^2 \ve^2 \beta(\beta+n-2) \int r^{-2\alpha +2\beta +2\ve -4} v^2
\\
& = - \frac{\alpha \ve}2 (2\alpha +4-n)(2\alpha+n)(2\alpha+2 +\ve(1-\frac{\alpha}2)) \int r^{-n +2\ve} v^2.  
\end{align*}
Combining the latter two observations we conclude that there exists a universal constant $C>0$ such that, if $0<\ve \le \frac{3}{20}$ as above, and $\alpha>1$ is sufficiently large depending on the dimension $n$, then
\begin{align}\label{v2}
&- 2 \alpha \ve (2\beta+n-1)  (2 \beta+ \ve + n-2) \int r^{-2\alpha + 2\beta+\ve -3} vv_r  + 2 \alpha^2 \ve^2 (2\beta+n-1)  \int r^{-2\alpha+2\beta+2\ve -3} vv_r 
\\
&- 4  \alpha \ve \beta(\beta+n-2) \int r^{-2\alpha+2\beta +\ve -3} vv_r 
\notag
\\
& -  2 \alpha \ve \beta(\beta+n-2) (2\beta+\ve+n-2)  \int r^{-2\alpha + 2\beta +\ve -4} v^2 
+ 2 \alpha^2 \ve^2 \beta(\beta+n-2) \int r^{-2\alpha +2\beta +2\ve -4} v^2
\notag
\\
& \geq - C \alpha^4 \ve \int r^{-n+\ve} v^2.  
\notag
\end{align}
We also note that by further restricting $\ve$, say $0<\ve\le \frac{1}{20}$, we can ensure by \eqref{choice} that for $\alpha>1$ we have
\begin{equation}\label{spt2}
4 \alpha \ve (2\beta+n-1) \int r^{-2\alpha + 2 \beta +\ve -2} v_r^2 = 4 \alpha\ve (2\alpha+3)  \int r^{-n+2+\ve} v_r^2 \leq \alpha^2 \int r^{-n+2} v_r^2.
\end{equation}
From \eqref{vr}, \eqref{spt2} and  \eqref{choice} once again, we infer the following estimate for $\alpha$ large enough (say, $\alpha\ge 10$)
\begin{align}\label{tyu}
& (2\beta+n-1)^2 \int r^{-2\alpha + 2\beta - 2} v_r^2 + 2 (2\beta+n-1)\int r^{-2\alpha + 2\beta -1} v_r v_{rr}
\\
& - 4 \alpha \ve (2\beta+n-1) \int r^{-2\alpha + 2 \beta +\ve -2} v_r^2
\notag\\
& \ge \big[(2\alpha+3)^2 - 2(2\alpha+3) - \alpha^2\big]\int r^{-n +2} v_r^2 \ge 2 \alpha^2 \int r^{-n +2} v_r^2.
\notag
\end{align}
The tenth integral in the right-hand side of \eqref{rt12nd} is simply handled as follows
\begin{align}\label{10}
& \left|2(2\beta+n-1) \int r^{-2\alpha+2\beta-1} v_t v_r\right|\le 4\alpha(1+\frac{3}{2\alpha})\int r^{-n+3} |v_t| |v_r|
\\
& \le 5\alpha\left(\frac{\alpha}5 \int  r^{-n+2} v_r^2 + \frac{5}{\alpha} \int  r^{-n+4} v_t^2\right) \le \alpha^2 \int  r^{-n+2} v_r^2 + 25 \int  r^{-n+4} v_t^2.
\notag
\end{align}

Finally, we handle the first integral in the right-hand side of \eqref{rt12nd}. We stress that such integral accounts for the critical term in the Carleman estimate \eqref{est1}. 
Recall that in the Fourier decomposition of $L^2(\sn)$, if $Y_k(\omega)$ is a  spherical harmonic of degree $k \in \mathbb{N}\cup\{0\}$ (that we assume normalised so that $\int_{\sn} |Y_k(\omega)|^2 d\omega = 1$), then $\Delta_{\sn} Y_k = - k(k+n-2) Y_k$. Therefore, if we write $v(x,t) = v(r\omega,t)$, and we indicate with $v_k(r,t) = \int_{\sn} v(r\omega,t) Y_k(\omega) d\omega$ its $k$-th Fourier coefficient in the Fourier decomposition $v(r\omega,t) = \sum_{k=0}^\infty v_k(r,t) Y_k(\omega)$, then we have
\[
\Delta_{\sn} v(r\omega,t) = - \sum_{k=0}^\infty k(k+n-2) v_k(r,t) Y_k(\omega).
\]
Using this representation and Parseval's theorem, we obtain 
\begin{align*}
& \int r^{-2\alpha+2\beta-4} (\beta(\beta+n-2) v + \Delta_{\sn}v )^2 
 = \int r^{-n}( \beta(\beta+n-2) v + \Delta_{\sn} v)^2
 \\
& =  \int_{0}^T \int_{0}^{\infty} r^{-1} \sum_{k=0}^{\infty} ( \beta(\beta+n-2) - k(k+n-2))^2 v_k (r,t) ^2 dr dt 
\end{align*}
At this point, we assume that $\text{dist}(\beta, \mathbb{N})= 1/2$. Since for every $k \in \mathbb{N}\cup\{0\}$ we have 
\[
 (\beta(\beta+n-2) - k(k+n-2))^2 = ((\beta-k) (\beta+k + n-2))^2\ge \frac 14 \left(\alpha+\frac{2k+n}2\right)^2,
\]
we thus infer  
\begin{align}\label{spt1}
& \int r^{-n}( \beta(\beta+n-2) v + \Delta_{\sn} v)^2 \ge  \frac{\alpha^2}{4}  \int_{0}^T \int_{0}^{\infty} r^{-1} \sum_{k=0}^{\infty} v_k (r,t) ^2 dr dt = \frac{\alpha^2}{4} \int r^{-n} v^2.
\end{align}
We now observe that, in view of \eqref{choice}, we have $\beta = k + \frac 12$ for some $k\in \mathbb N$ if and only if $\alpha = k-2 + \frac{n+1}2$. This shows that for $\alpha$ large (depending on $n$)
\begin{equation}\label{alpha}
\text{dist}(\beta, \mathbb{N})=\frac{1}{2}\ \Longleftrightarrow\ \alpha = k + \frac{n+1}{2},\ \ \text{for some}\ \ k\in \mathbb N.
\end{equation}
Combining \eqref{rt12nd}-\eqref{spt1}, we reach the conclusion that there exists a universal $C>0$ such that for all $\alpha>1$ sufficiently large (depending on $n$) and as in \eqref{alpha}, and for every $0<\ve\le \frac{1}{20}$, one has  
\begin{align*}
& \int r^{-2\alpha}  e^{2\alpha r^\ve} (\Delta u - u_t)^2 \geq \frac{\alpha^2}{4}  \int r^{-n} v^2 +  2\alpha^2  \int r^{-n+2} v_r^2 + \frac{37}{10} \alpha^2 \ve \int r^{-n+\ve} |\nabla_{\sn} v|^2 
\\
& - \alpha^2 \int  r^{-n+2} v_r^2 - 25 \int  r^{-n+4} v_t^2 - C\alpha^4 \ve \int r^{-n+\ve} v^2
\\
& = \frac{\alpha^2}{4}  \int r^{-n} v^2 +  \alpha^2  \int r^{-n+2} v_r^2 + \frac{37}{10} \alpha^2 \ve \int r^{-n+\ve} |\nabla_{\sn} v|^2   - 25 \int  r^{-n+4} v_t^2 - C\alpha^4 \ve \int r^{-n+\ve} v^2.
 \\
 & \ge \frac{\alpha^2}{4}  \int r^{-n} v^2 - 25 \int  r^{-n+4} v_t^2 - C\alpha^4 \ve \int r^{-n+\ve} v^2.
\end{align*}
Recalling \eqref{choice}, and that $v = r^{-\beta} e^{\alpha r^\ve} u$, we conclude that we have  established the following bound
\begin{align}\label{hold1}
& \int r^{-2\alpha}  e^{2\alpha r^\ve} (\Delta u - u_t)^2 \geq \frac{\alpha^2}{4}  \int r^{-2\alpha -4} e^{2\alpha r^\ve} u^2
\\
&  - 25 \int  r^{-2\alpha} e^{2\alpha r^\ve} u_t^2 - C\alpha^4 \ve \int r^{-2\alpha - 4 +\ve} e^{2\alpha r^\ve} u^2.
\notag
\end{align}
Keeping in mind that our final objective is proving \eqref{est1}, we mention at this point that the two negative terms in the right-hand side of  \eqref{hold1} represent a series obstruction toward such goal. To overcome such difficulty we will establish the following delicate a priori bound. We stress that, differently from \eqref{hold1}, the spectral gap assumption  \eqref{alpha} is not needed. 

\begin{lemma}\label{L:del} 
Let $R<1$ and let  $u \in C_0^{\infty}((B_R \setminus \{0\}) \times (0,T))$. There exist constants $C_0= C_0(n)>0$, $\alpha(n)>>1$ and $0<\ve(n)<<1$, such that for all $\alpha\ge \alpha(n)$ and every $0<\ve<\ve(n)$ one has 
\begin{align}\label{claim1}
\frac{C_0}{\alpha} \int r^{-2\alpha} e^{2\alpha r^\ve} u_t^2 + C_0\alpha^3\ve^2 \int r^{-2\alpha-4+\ve} e^{2\alpha r^\ve} u^2   \leq \int r^{-2\alpha}  e^{2\alpha r^\ve} (\Delta u - u_t)^2.
\end{align}
\end{lemma}

The proof of Lemma \ref{L:del} is postponed to the end of the section. With such result in hands we now proceed to complete the proof of Theorem \ref{thm2}.
We fix $0<\ve(n)<1$ and $\alpha(n)>>1$ such that \eqref{hold1} and \eqref{claim1} hold simultaneously for $0<\ve<\ve(n)$ and $\alpha>\alpha(n)$ and satisfying \eqref{alpha}.  We then choose and fix $\ve \in (0, \ve(n))$. Corresponding to such a  choice of $\ve$, we now select $C_2 = C_2(n,\ve)>1$ such that
 \begin{equation*}
 C_2 C_0  \ve  \geq 2 C \ \ \ \text{and}\ \ \ C_2 C_0 > 25.
 \end{equation*}
With such constant $C_2$ in hands, we multiply \eqref{claim1} by $C_2\alpha$ and add the resulting inequality to \eqref{hold1}, obtaining
 \begin{align}\label{hold2}
& \frac{\alpha^2}{4}  \int r^{-2\alpha - 4} e^{2\alpha r^\ve} u^2 + (C_2 C_0 \ve -C)\alpha^4 \ve  \int r^{-2\alpha - 4 +\ve} e^{2\alpha r^\ve} u^2
\\
& + (C_2 C_0 - 25 )\int r^{-2\alpha} e^{2\alpha r^\ve} u_t^2  \le (C_2 \alpha +1 )\int r^{-2\alpha} e^{2\alpha r^\ve} (\Delta u - u_t)^2.
\notag
\end{align}
By our choice of $C_2$, and after dividing through by $\alpha$, the following inequality easily follows from \eqref{hold2}  
\begin{align*}
& \alpha  \int r^{-2\alpha - 4} e^{2\alpha r^\ve} u^2 + \alpha^3  \int r^{-2\alpha - 4 +\ve} e^{2\alpha r^\ve} u^2
\le \frac{2 C_2}{\min\{1/4,C\ve\}}  \int r^{-2\alpha} e^{2\alpha r^\ve} (\Delta u - u_t)^2.
\end{align*}
Modulo Lemma \ref{L:del}, this completes the proof of the Carleman estimate \eqref{est1}. 
 
\end{proof}

We now turn to the 
\begin{proof}[Proof of Lemma \ref{L:del}]
The proof of the estimate \eqref{claim1} is somewhat delicate. Letting as before $v=r^{-\beta} e^{\alpha r^\ve}  u$, at first we write the expression of  $\Delta u - u_t$ in \eqref{comp1} in the form
\[
\Delta u- u_t = a+ b,
\]
with 
\[
a= r^{\beta}  e^{-\alpha r^\ve} (v_{rr} + B(r, \alpha, \beta) v  + r^{-2} \Delta_{\sn} v ),
\]
and
\[
b= r^{\beta} e^{-\alpha r^{\ve}}( A(r, \alpha, \beta) v_r - v_t),
\]
  where 
\begin{equation}\label{pol0}
\begin{cases}
B(r, \alpha, \beta)=\left(\alpha^2 \ve^2 r^{2\ve-2} + \beta(\beta+n-2)r^{-2}- \alpha\ve(2\beta+\ve+ n-2) r^{\ve-2}\right),
\\
\\
A(r, \alpha, \beta)=(2\beta+n-1) r^{-1} - 2 \alpha \ve r^{\ve-1}.
\end{cases}
\end{equation}  
In what follows, for the sake of brevity we simply write $A$ and $B$, instead of $A(r, \alpha, \beta)$ and $B(r, \alpha, \beta)$. Also, the equation \eqref{choice} will be repeatedly used without further reference to it. For instance, we note that, using such equation, we have from \eqref{pol0} the following alternative expression of $A$ and $B$
\begin{equation}\label{pol}
\begin{cases}
B =\left(\alpha^2 \ve^2 r^{2\ve-2} + \alpha^2 (1+o(1)) r^{-2}- 2\alpha^2 \ve(1+o(1)) r^{\ve-2}\right),
\\
\\
A = 2\alpha (1+o(1)) r^{-1} - 2 \alpha \ve r^{\ve-1},
\end{cases}
\end{equation} 
where we have denoted with $o(1)$ quantities which do not depend of $r$ and such that $|o(1)| \to 0$ as $\alpha \to \infty$.
The reader should note that among the ensuing computations that lead  to \eqref{claim1}, some are similar to those of  \eqref{rt12nd}-\eqref{spt1}, and therefore several details will be skipped.   

With this being said, our strategy is to expand  $(\Delta u - u_t)^2 =a^2 + b^2 + 2ab$, 
and then estimate from below  each of the corresponding integrals 
$\int r^{-2 \alpha} e^{2\alpha r^\ve} a^2$, $\int r^{-2 \alpha} e^{2\alpha r^\ve}  b^2$, and  $2\int r^{-2 \alpha} e^{2\alpha r^\ve} ab$
  in an appropriate way. 
We begin with 
\begin{align}\label{sepr}
& 2\int r^{-2\alpha} e^{2\alpha r^\ve} ab = 2 \int A r^{-2\alpha + 2\beta}  v_r v_{rr} + 2  \int A B r^{-2\alpha+ 2\beta}  vv_r
\\
& + 2 \int A r^{-2\alpha+2 \beta-2}  v_r \Delta_{\sn} v - 2 \int r^{-2\alpha+2\beta} v_t (v_{rr} + B v + r^{-2} \Delta_{\sn} v),\notag
\end{align}
and estimate each term that appears in the right-hand side of \eqref{sepr} separately. 
By Lemma \ref{L:radial} we have
\begin{align}\label{vr1}
&2 \int A r^{-2\alpha + 2\beta}  v_r v_{rr} = -2(2\alpha+3)\int \int r^{-n+2} v_r^2 
\\
& + 2\alpha \ve (2-\ve)\int \int r^{-n+2+\ve} v_r^2 \geq - 5 \alpha \int r^{-n+2} v_r^2,
\notag
\end{align}
provided $\alpha\ge 6$. Next, we find from \eqref{pol} 
\begin{align}\label{o1}
& 2 \int  A B r^{-2\alpha + 2\beta} v v_r = 
 4\alpha^3(1+o(1)) \int r^{-n+1} v v_r + 12\alpha^3 \ve^2 (1+o(1)) \int r^{-n+1+2\ve} vv_r
 \\
 & - 12 \alpha^3 \ve (1+o(1)) \int r^{-n+1+\ve} vv_r - 4 \alpha^3 \ve^3 \int r^{-n+1+3\ve} v v_r.
\notag
\end{align}
As in \eqref{1}, we have 
\begin{align*}
\int r^{-n+1} vv_r=0.
\end{align*}
A repeated application of Lemma \ref{L:radial}, and the fact that $r\le R<1$, give
\[
2 \int r^{-n+1+2\ve} vv_r \ge - 2\ve \int r^{-n+\ve} v^2,
\]
\[
- 2 \int r^{-n+1+\ve} vv_r = \ve \int r^{-n+\ve} v^2,
\]
and
\[
- 2 \int r^{-n+1+3\ve} vv_r =  3 \ve \int r^{-n+3\ve} v^2 \ge 0.
\]
Using the latter four relations in \eqref{o1}, we conclude
\begin{align*}
& 2 \int  A B r^{-2\alpha + 2\beta} v v_r \ge \left[6 \alpha^3 \ve^2 (1+o(1)) - 12\alpha^3 \ve^3 (1+o(1))\right] \int r^{-n+\ve} v^2
\\
& = 6 \alpha^3 \ve^2 (1-2\ve)(1+o(1))\int r^{-n+\ve} v^2.
\end{align*}
It is clear from this estimate that, if $0<\ve<\frac{1}{240}$, then for $\alpha>>1$ sufficiently large we have
\begin{equation}\label{lb1}
2 \int  A B r^{-2\alpha + 2\beta} v v_r \ge \frac{59}{10} \alpha^3 \ve^2 \int r^{-n+\ve} v^2.
\end{equation}
Furthermore, we obtain from \eqref{pol}
\begin{align}\label{undes}
& 2 \int A r^{-2\alpha+2 \beta-2} v_r \Delta_{\sn} v = 4\alpha(1+o(1)) \int r^{-n+1} v_r \Delta_{\sn} v 
\\
&  - 4\alpha \ve  \int r^{-n+1+\ve} v_r \Delta_{\sn} v =- 2\alpha \ve^2  \int r^{-n+\ve} |\nabla_{\sn} v|^2,
\notag
\end{align}
where in the last equality we have used \eqref{zero} and \eqref{sp1}. We also have
\begin{align*}
&- 2 \int r^{-2\alpha+2\beta} v_t (v_{rr} + B v + r^{-2} \Delta_{\sn} v) = - 2 \int r^{-n+4} v_t v_{rr}
\\
& - 2 \int B r^{-n+4} v v_t - 2 \int r^{-n+2} v_t \Delta_{\sn} v. 
\end{align*}  
The latter two integrals in the right-hand side vanish, similarly to the two computations following \eqref{choice}. Using instead Lemma \ref{L:radial}, we find
\begin{align*}
& -2 \int r^{-n+4} v_{rr} v_t =  6 \int r^{-n+3} v_t v_r + 2 \int r^{-n+4} v_r v_{tr}
\\
&= 6 \int r^{-n+3} v_t v_r +  \int \partial_t (r^{-n+4} v_r^2) = 6 \int r^{-n+3} v_t v_r.
\end{align*}
We deduce that 
\begin{align}\label{vg50}
&- 2 \int r^{-2\alpha+2\beta} v_t (v_{rr} + B v + r^{-2} \Delta_{\sn} v)=  6 \int r^{-n+3} v_t v_r.
\end{align}
If we now combine \eqref{sepr}-\eqref{vg50} we conclude that 
\begin{align}\label{ab}
& 2 \int  r^{-2\alpha} e^{2\alpha r^{2\ve} }ab \geq \frac{59}{10} \alpha^3 \ve^2 \int r^{-n+\ve} v^2 + 6 \int r^{-n+3} v_t v_r
\\
& - 2\alpha \ve^2 \int r^{-n+\ve} |\nabla_{\sn} v|^2 - 5 \alpha \int r^{-n+2} v_r^2.
\notag
\end{align}
Our next objective is to eliminate  the  negative term   $ - 2\alpha \ve^2 \int r^{-n+\ve} |\nabla_{\sn} v|^2$  in the right hand side of \eqref{ab}.  We stress that, although at a first glance it might seem that such integral could be absorbed by the above discarded positive term $\frac{37}{10} \alpha^2 \ve \int r^{-n+\ve} |\nabla_{\sn} v|^2 $ in \eqref{hold1}, a careful look at the analysis that led to \eqref{hold2} reveals that this would not work.

Having said this, to accomplish our objective we instead proceed with estimating from below $\int r^{-2\alpha} e^{2\alpha r^\ve} a^2$ as follows 
\begin{align}\label{nm}
& \int r^{-2\alpha} e^{2\alpha r^\ve} a^2= \int r^{-2\alpha+2\beta}    ( v_{rr} + Bv  + r^{-2} \Delta_{\sn} v )^2 
\\
& = \int r^{-2\alpha+2\beta} \left( ( v_{rr} + B v+ r^{- 2} \Delta_{\sn} v + \alpha \ve^2   r^{\ve-2} v) - \alpha \ve^2   r^{\ve-2} v\right)^2
\notag
\\ 
& \geq - 2 \alpha \ve^2   \int r^{-n+2+\ve} v \left(  v_{rr} + (B + \alpha \ve^2 r^{\ve-2}) v + r^{- 2} \Delta_{\sn}v\right) 
\notag
\end{align}
where in the last inequality above we have used $(c_1+c_2)^2 \geq 2c_1c_2$, with 
\[
c_1=  v_{rr} + (B + \alpha \ve^2 r^{\ve-2}) v+ r^{- 2} \Delta_{\sn} v,\ \ \ \ \ \ \ \ \ \ 
c_2= - \alpha \ve^2   r^{\ve-2} v.
\]
We then estimate each term in the right-hand side of \eqref{nm} as follows.  First, the divergence theorem on $\sn$ gives
\begin{equation}\label{kl1}
-2 \alpha \ve^2 \int  r^{-n +\ve} v \Delta_{\sn} v=  2 \alpha \ve^2 \int r^{-n+\ve} |\nabla_{\sn} v|^2,
\end{equation}
which precisely eliminates the negative term $- 2\alpha \ve^2 \int r^{-n+\ve} |\nabla_{\sn} v|^2$ in  \eqref{ab}, see \eqref{tough1} below. Secondly, a repeated application of Lemma \ref{L:radial} gives
\begin{equation}\label{kl2}
-2 \alpha  \ve^2  \int r^{-n+2 +\ve} v v_{rr} = 2 \alpha \ve^2 \int r^{-n+2+\ve} v_r^2 - \alpha \ve^3 (1+\ve) \int r^{-n+\ve} v^2. 
\end{equation}
Thirdly, using the expression of $B$ in  \eqref{pol} it is easily seen that for $\alpha>>1$ sufficiently large and $0< \ve < \frac{1}{240}$ we have
\begin{equation}\label{kl4}
-2\alpha \ve^2  \int r^{-n+2+\ve} (B  +\alpha \ve^2 r^{\ve -2} )   v^2  \geq  - 3 \alpha^3 \ve^2 \int  r^{-n+\ve} v^2.
\end{equation}
From \eqref{nm}-\eqref{kl4} we infer that
\begin{equation}\label{kl5}
\int r^{-2\alpha} e^{2\alpha r^\ve}   a^2  \geq 2 \alpha \ve^2 \int r^{-n+\ve} |\nabla_{\sn} v|^2  - \alpha \ve^3 (1+\ve) \int r^{-n+\ve} v^2 - 3 \alpha^3 \ve^2 \int  r^{-n+\ve} v^2. 
\end{equation}
Combining now \eqref{ab} and \eqref{kl5}, and keeping in mind that $b = r^\beta e^{-\alpha r^{\ve}}(A v_r - v_t)$,
we obtain 
\begin{align}\label{tough1}
& \int r^{-2\alpha} e^{2\alpha r^{\ve}} (\Delta u- u_t)^2 \geq \left( \frac{29}{10} \alpha^3 \ve^2 - \alpha \ve^3 (\ve+1) \right)  \int r^{-n+\ve} v^2 - 5 \alpha \int r^{n+2} v_r^2 
\\
& + 6 \int r^{-n+3} v_t v_r + \int r^{-n+2} (r A v_r - r v_t)^2.  
\notag
\end{align}  
Now, for $ 0<\ve< \frac{1}{10}$ we have
\[
\frac{29}{10} \alpha^3 \ve^2 - \alpha \ve^3 (\ve+1) \geq 2 \alpha^3 \ve^2.
\]
Using this in \eqref{tough1} above, we deduce the following estimate
\begin{align}\label{tough}
& \int r^{-2\alpha} e^{2\alpha r^{\ve}} (\Delta u- u_t)^2 \geq 2 \alpha^3 \ve^2 \int r^{-n+\ve} v^2 - 5 \alpha \int r^{n+2} v_r^2 
\\
& + 6 \int r^{-n+3} v_t v_r + \int r^{-n+2} (r A v_r - r v_t)^2.  
\notag
\end{align}

Our next and final objective is to bound from below the right-hand side of \eqref{tough}  by an expression that includes $\frac{C_0}{\alpha} \int r^{-n+4} v_t^2$, as desired in \eqref{claim1}. In the process, we also need to control the negative term $- 5 \alpha \int r^{-n+2} v_r^2$. 
To achieve this we consider the last two terms in the right-hand side of \eqref{tough} and, adapting a similar idea in  \cite{EV}, \cite{V}, we estimate them in two different ways. First, we proceed as follows 
\begin{align*}
\mathcal{R} & \overset{def}{=} 6 \int r^{-n+3} v_t v_r + \int r^{-n+2} (r A v_r - r v_t)^2 
\\
& = 6 \int r^{-n+3} v_t v_r + \int r^{-n+2} ((r A - 3) v_r - r v_t + 3v_r)^2
\\
& =  \int r^{-n+2} (6 rA -9)v_r^2 + \int r^{-n+2} ((r A - 3) v_r - r v_t)^2
\\
& \ge \int r^{-n+2} (6 rA -9)v_r^2.
\end{align*}
Using \eqref{pol}, it is easy now to recognise that for $\alpha>>1$ large and $0 < \ve < \frac{1}{300}$, we have $r A \geq  \frac{19}{10} \alpha$, and therefore 
$(6 rA -9) \ge \frac{114}{10}  \alpha -9 \ge \frac{112}{10} \alpha$,
by increasing further the value of $\alpha>>1$. In conclusion, we obtain
\begin{equation}\label{ty1}
\mathcal{R} \geq  \frac{112}{10} \alpha \int r^{-n+2} v_r^2.
\end{equation}
Proceeding in a similar way, we recognise that 
\begin{align*}
\mathcal{R} & = 6 \int r^{-n+3} v_t v_r + \int r^{-n+2} \left(r A v_r - r\left(1-\frac{3}{rA}\right)v_t - \frac{3}{A}v_t\right)^2
\\
& = 6 \int r^{-n+3} v_t v_r + \int r^{-n+4} \left(A v_r - \left(1-\frac{3}{rA}\right)v_t\right)^2 
\\
& + 9 \int r^{-n+2} \frac{1}{A^2} v_t^2 - 2\int r^{-n+3} \left(Av_r - \left(1-\frac{3}{rA}\right)v_t\right) \frac{3}{A} v_t
\\
& \ge \int r^{-n+4} \frac{1}{(rA)^2} \left(6 r A - 9\right) v_t^2.
\end{align*}
At this point we note that \eqref{pol} implies 
$r A \leq 2\alpha (1+o(1)) \le 3 \alpha$, 
 for $\alpha>>1$ sufficiently large. Combining this with the previous estimate from below $6 rA -9 \ge \frac{112}{10} \alpha$, we obtain $\frac{1}{(rA)^2} \left(6 r A - 9\right) \ge \frac{1}{\alpha}$, for $\alpha>>1$ large and $0 < \ve < \frac{1}{300}$. This gives
\begin{equation}\label{ty2}
\mathcal{R} \geq \frac{1}{\alpha} \int r^{-n+4} v_t^2.
\end{equation}
If we now split $\mathcal R= \frac{9}{10} \mathcal R + \frac{1}{10}\mathcal R$,
 and apply \eqref{ty1} to $\frac{9}{10}\mathcal R $, and \eqref{ty2} to $\frac{1}{10}\mathcal R$, we finally obtain
 \begin{align}\label{ty4}
 & \mathcal{R} \geq \frac{1008}{100} \alpha \int r^{-n+2} v_r^2  +\frac{1}{10\alpha}  \int r^{-n+4} v_t^2.
 \end{align} 
At this point we are almost done. Using
the inequality \eqref{ty4} in \eqref{tough}, we obtain
\begin{align*}
& \int r^{-2\alpha} e^{2\alpha r^{\ve}} (\Delta u- u_t)^2 \geq 2 \alpha^3 \ve^2 \int r^{-n+\ve} v^2 + \left(\frac{1008}{100}- 5\right) \alpha \int r^{-n+2} v_r^2 
 +\frac{1}{10\alpha}  \int r^{-n+4} v_t^2
 \\
 & \ge 2 \alpha^3 \ve^2 \int r^{-n+\ve} v^2 + +\frac{1}{10\alpha}  \int r^{-n+4} v_t^2,  
\end{align*}
where we have taken advantage of the crucial gain in positivity of the coefficient of $\int r^{-n+2} v_r^2$. If we now keep in mind that $v=r^{-\beta} e^{\alpha r^\ve}  u$, we finally deduce that  \eqref{claim1} holds.

 \end{proof}
 

\section{Proof of Theorem \ref{thm1}}\label{S:2}

In this section we show how to obtain the sucp result in Theorem \ref{thm1} from  Theorem \ref{thm2}. With the new estimate \eqref{est1} in hands, we can adapt to the critical differential inequality \eqref{par} some of the ideas that in \cite[Theor. 15, p. 658-664]{V} were developed in the subcritical context of \eqref{v}. As we have mentioned in the introduction, this entails a delicate modification of Vessella's proof. For this reason, and for the sake of the reader's comprehension, we will present a detailed account.  
We begin with the following simple Caccioppoli type inequality.
  
\begin{lemma}\label{cac}
Let $u$ be a solution to \eqref{par}  in $B_R \times (-T, T)$ and let $0 < a< 1 < b$. Then, there exists a constant $C_1>0$, depending on $n, a, b, T$ and $M$ in \eqref{par}, such that for every $r< \min\{1, R\}$ the following holds 
\[
\int_{ \{r/2<|x| < r\} \times (-T/2, T/2)} |\nabla u|^2 \leq \frac{C_1}{r^2} \int_{ \{r(1-a)/2 < |x|<  b r\} \times (-T, T) } u^2.
\]
\end{lemma}

\begin{proof}
From \eqref{par}, we may assume that $u$ solves
$\Delta u - u_t = Vu$, 
where
$ |V(x,t)| \leq \frac{M}{|x|^2}$. Let now $ \phi \equiv 1$ in $\{r/2 < |x| < r\} \times (-T/2, T/2)$,
and  vanishing outside $\{r(1-a)/2 < |x|<  b r\} \times (-T, T)$. 
Using $\phi^2 u$ as a test function in the weak form of the equation we obtain
\begin{align}\label{st0}
\int |\nabla u|^2 \phi^2 +  \int  uu_t \phi^2 \leq  2 \int  |\nabla u| |\nabla \phi| |\phi| |u| + \int |V| u^2  \phi^2 
\end{align}
Since an integration by parts gives 
\[
\int uu_t \phi^2= \frac{1}{2} \int (u^2)_t \phi^2 = - \int u^2 \phi \phi_t,
\]
we obtain from \eqref{st0}
\[
\int |\nabla u|^2 \phi^2  \leq  2 \int  |\nabla u| |\nabla \phi| |\phi| |u| +  \int |V| u^2  \phi^2 +  \int u^2 |\phi| |\phi_t|.
\]
By the Cauchy-Schwarz inequality we have in a standard fashion
$2 \int  |\nabla u| |\nabla \phi| |\phi| |u|  \leq \frac{1}{2} \int |\nabla u|^2 \phi^2 +  2 \int u^2 |\nabla \phi|^2$. Substitution in the latter inequality gives
\begin{align}\label{st1}
\int |\nabla u|^2 \phi^2  \leq  4 \int  u^2 |\nabla \phi|^2  +  2 \int |V| u^2  \phi^2 + 2 \int u^2 |\phi| |\phi_t|.
\end{align}
Using the bounds $|\nabla \phi| \leq C_2/|x|$, $|\phi_t| \leq C_3/T$, and the fact that $\phi, \nabla \phi, \phi_t$ are supported in $\{r(1-a)/2 < |x|<  b r\} \times (-T, T)$, we obtain from \eqref{st1} that the following holds,
\[
\int |\nabla u|^2 \phi^2 \leq \frac{C_1}{r^2} \int_{ \{r(1-a)/2 < |x|<  b r\} \times (-T, T) } u^2,
\]
for some $C_1$ depending on $n, a, b, T$ and $M$. The desired conclusion   follows by bounding from below the integral in the left-hand side with one over the region where $\phi \equiv 1$. 

\end{proof}

\begin{proof} [Proof of Theorem \ref{thm1}]
Throughout the proof the letter $C$ will indicate an all purpose constant which might change from line to line, and which could depend in some occurrences on the number $T$. In what follows it will be easier for the computations if we work with the symmetric time-interval $(-T,T)$, instead of $(0,T)$.  Without loss of generality, we also assume that $R<1$.  Let $0 < r_1 < r_2/2 < 4r_2 <r_3 < R/2$ be fixed, and let $\phi(x)$ be a smooth function such that    
$\phi(x) \equiv 0$ when $|x| < r_1/2$,
$\phi(x) \equiv 1$ when $r_1 < |x|< r_2$,
$\phi(x) \equiv 0$ when $|x|> r_3$. We now let $T_2= T/2$ and $T_1= 3T/4$, so that $0<T_2<T_1<T$. As in  \cite{V}, we let $\eta(t)$ be a smooth even function such that $\eta(t) \equiv 1$ when $|t| < T_2$, $\eta(t) \equiv 0$, when $|t| > T_1$. Furthermore, it will be important in the sequel (see \eqref{bn1} below) that $\eta$  decay exponentially near $t = \pm T_1$. As in (118) of \cite{V} we take  
\begin{equation}\label{d4}
\eta(t)= \begin{cases} 0\ \ \ \ \ \ \ \ \ -T\le t\le -T_1
\\
\exp \left(-\frac{T^3(T_2+t)^4}{(T_1 +t)^3(T_1-T_2)^4} \right)\ \ \ \ \ \ \ -T_1\le t \le -T_2,
\\
1,\ \ \ \ \ \ \  t \in -T_2\le t \le 0.
 \end{cases}
 \end{equation}
We suppose that $u$ parabolically vanishes to infinite order in the sense of \eqref{vp}, and we want to conclude that $u \equiv 0$ in $B_R \times (-T, T)$. We assume that this not the case and show that we reach a contradiction. 
Without loss of generality we can (and will) assume that
 \begin{equation}\label{assume}
  \int_{ B_{r_2} \times (-T_2, T_2)} u^2 \neq 0.
  \end{equation}
Otherwise, the result in \cite{V} implies $u \equiv 0$  in $B_R \times (-T_2, T_2)$ and by the arguments that follow we could conclude that $u \equiv 0$ also  for $|t| > T_2$. The assumption \eqref{assume} will be used in the very end in the equation \eqref{choice2}. 

Now, with $u$ as in Theorem \ref{thm1} we let  $v= \phi \eta u$. A standard limiting argument allows to use such $v$ in the Carleman estimate \eqref{est1}, obtaining
\[
\alpha \int r^{-2\alpha-4}  e^{2\alpha r^\ve} v^2  + \alpha^3 \int r^{-2\alpha - 4+\ve} e^{2\alpha r^\ve} v^2   \leq  C \int r^{-2\alpha}  e^{2\alpha r^\ve} (\Delta v - v_t)^2. 
\]
Here, we have fixed some $\ve \in (0, \ve(n))$, where $\ve(n)$ is as in the hypothesis of  Theorem \ref{thm2}. 
Keeping  in mind that 
\[
\Delta v - v_t = \phi \eta (\Delta u - u_t) + u (\eta \Delta \phi - \phi \eta_t) + 2 \eta <\nabla u,\nabla \phi>,
\]
 we obtain
 \begin{align}\label{et1}
& \alpha \int r^{-2\alpha- 4} e^{2\alpha r^\ve} v^2  + \alpha^3 \int r^{-2\alpha -  4+\ve} e^{2\alpha r^\ve} v^2 \leq  C \int r^{-2\alpha} e^{2\alpha r^\ve} (\Delta u - u_t)^2 \phi^2 \eta^2 
\\
& + C  \int r^{-2\alpha} e^{2\alpha r^\ve}  ( |\nabla u|^2 |\nabla \phi|^2 + u^2 (\Delta \phi) ^2 ) \eta^2 + u^2 \phi^2 \eta_t^2.\notag
\end{align}
Since the differential inequality \eqref{par} gives
\[
C \int r^{-2\alpha} e^{2\alpha r^\ve}(\Delta u - u_t)^2 \phi^2 \eta^2 \leq C M^2 \int r^{-2\alpha - 4} e^{2\alpha r^\ve} v^2,
\]
if we choose 
\begin{equation}\label{choice1}
\alpha \ge  2 C M^2,
\end{equation}
then the first integral in the right-hand side of \eqref{et1} can be absorbed in the left-hand side. Consequently, from the way $\phi$ and $\eta$ have been chosen, and bearing in mind that $\nabla \phi$ is supported in $\{r_1/2  < r< r_1\} \cup \{r_2  < r < r_3\}$ and that we have in such a set
$|\nabla \phi | = O(1/r), |\Delta \phi|= O(1/r^2)$, we obtain from \eqref{et1} 
\begin{align}\label{et22}
& \alpha \int r^{-2\alpha- 4} e^{2\alpha r^\ve} v^2  
 \leq C \int_{  \{r_1/2 < r < r_1\} \times (-T_1,T_1)  } e^{2\alpha r^\ve} ( r^{-2\alpha-4} u^2  + r^{-2\alpha-2} |\nabla u|^2)
\\
&   + C  r_2^{-2\alpha-4} e^{2\alpha r_2^\ve} \int_{  \{r_2 < r < r_3\} \times (-T_1, T_1)  } (u^2 +|\nabla u|^2 ) 
\notag\\
& + C \int_{ \{r_1/2 < |x| < r_3\} \times (-T_1, T_1) } r^{-2\alpha} e^{2\alpha r^\ve}    u^2 \phi^2  \eta_t^2 - 2 \alpha^3 \int r^{-2\alpha - 4+\ve} e^{2\alpha r^\ve}  v^2.
\notag
\end{align}
In \eqref{et22},  we have also  used the fact that, since  the functions 
\begin{equation}\label{monotone}
r \to r^{-2\alpha-4} e^{2\alpha r^\ve} \ \ \ \ \ \ \text{and}\  \ \ \ \ \ \  r \to r^{-2\alpha} e^{2\alpha r^\ve}  
\end{equation}
are decreasing in $(0,1)$, we can estimate 
\begin{align*}
& \int_{  \{r_2 < r < r_3\} \times (-T_1, T_1)  } r^{-2\alpha} e^{2\alpha r^\ve} (u^2 (\Delta \phi)^2 +|\nabla u|^2 |\nabla \phi|^2)
\\
& \leq C r_2^{-2\alpha-4} e^{2\alpha r_2^\ve} \int_{  \{r_2 < r < r_3\} \times (-T_1, T_1)  } (u^2 +|\nabla u|^2 ). 
\end{align*}
 We now split the second to the last term in the right-hand side of \eqref{et22} in three parts
\begin{align}\label{spt}
& C \int_{ \{r_1/2 < |x| < r_3\} \times (-T_1,T_1) } r^{-2\alpha} e^{2\alpha r^\ve}   u^2 \phi^2  \eta_t^2= C \int_{\{r_1/2 < |x| < r_1\} \times (-T_1,T_1)} r^{-2\alpha} e^{2\alpha r^\ve}  u^2 \phi^2 \eta_t^2 
\\
& + C \int_{\{r_1 < |x| < r_2\} \times (-T_1,T_1)} r^{-2\alpha}  e^{2\alpha r^\ve} u^2 \phi^2 \eta_t^2 + C \int_{\{r_2 < |x| < r_3\} \times (-T_1, T_1)} r^{-2\alpha} e^{2\alpha r^\ve}   u^2 \phi^2 \eta_t^2.
\notag
\end{align}
Since $|\eta_t| \leq C/T$, the first and third terms in the right-hand side of \eqref{spt} are respectively estimated as follows using \eqref{monotone}
\begin{align}\label{g1}
&C \int_{\{r_1/2 < |x| < r_1\} \times (-T_1,T_1)} r^{-2\alpha} e^{2\alpha r^\ve}  u^2 \phi^2 \eta_t^2  \leq  C \left(\frac{r_1}{2} \right)^{-2\alpha} e^{ \frac{2 \alpha r_1}{2}} \int_{\{r_1/2 < |x| < r_1\} \times (-T_1,T_1)}  u^2\\
& \leq C \left( \frac{r_1}{2C_1}\right) ^{-2\alpha}\int_{\{r_1/2 < |x| < r_1\} \times (-T_1,T_1)}  u^2\notag\end{align}
 where in the last inequality in \eqref{g1}, we have used the fact that  $e^{ \frac{2 \alpha r_1}{2}} \leq e^{2\alpha r_3} \leq  C_1^{2\alpha}$, for some  $C_1>0$ depending only on $r_3$, which has been fixed.  Similarly, we have
 \begin{equation}\label{g2}
C \int_{\{r_2 < |x| < r_3\} \times (-T_1,T_1)} r^{-2\alpha} e^{2\alpha r^\ve}  u^2 \phi^2 \eta_t^2  \leq  C r_2^{-2\alpha} e^{2\alpha r_2^\ve}  \int_{\{r_2< |x| < r_3\} \times (-T_1,T_1)}  u^2.
\end{equation}
In order to estimate the second term in the right-hand side of \eqref{spt}, we combine it with the last integral in the right-hand side of \eqref{et22} and observe that, since $\phi \equiv 1$ in the region $\{r_1 < |x| < r_2\}$, and the function $\eta_t$ is supported in the set $(-T_1,-T_2) \cup (T_2, T_1)$, if we indicate $U = \{r_1 < |x| < r_2\} \times [(-T_1, -T_2) \cup (T_2, T_1) ]$, we can bound
\begin{align*}
&C \int_{\{r_1 < |x| < r_2\} \times (-T_1,T_1)} r^{-2\alpha} e^{2\alpha r^\ve} u^2 \phi^2 \eta_t^2 - 2 \alpha^3  \int r^{-2\alpha-4+\ve} e^{2\alpha r^\ve}  v^2\\
&\leq \int_{U} r^{-2\alpha-4+\ve} e^{2\alpha r^\ve} u^2 \eta^2 \left(C r^{4-\ve}   \frac{\eta_t^2}{\eta^2} - 2 \alpha^3 \right)
\\
& \leq \int_{U} r^{-2\alpha-4+\ve} e^{2\alpha r^\ve} u^2 \eta^2 \left(C r^{3}   \frac{\eta_t^2}{\eta^2} - 2 \alpha^3 \right)\notag \end{align*}
Note that in the last inequality  above, we used that for $\ve < 1$, we  have  $r^{4-\ve} < r^3$, for all $r<1$. 
At this point our objective is to establish the following estimate
\begin{align}\label{c8}
&\int_{U}  r^{-2\alpha-4+\ve} e^{2\alpha r^\ve}  u^2 \eta^2 \left( C r^3  \frac{\eta_t^2}{\eta^2} - 2\alpha^3 \right) \leq  C \int_{B_R \times (-T,T)} u^2. 
\end{align}
The proof of \eqref{c8} will be accomplished in several steps. First, we note that it suffices to concern ourselves with the portion of the integral in the left-hand side of \eqref{c8} over the region $U^- = \{r_1 < |x| < r_2\} \times (-T_1, -T_2)$, since the estimate on $U^+ = \{r_1 < |x| < r_2\} \times (T_2,T_1)$ is similar. Now, if $-T_1\le t\le -T_2$, keeping in mind that $T_1 -T_2 = \frac T4$, $|T_2 + t| \le T_1 - T_2 = \frac T4$, and that $\frac 34 T \le 4T_1- 3T_2+t \le T$, from \eqref{d4} a standard calculation shows
\begin{equation*}\label{y2}
\bigg|\frac{\eta_t}{\eta}\bigg| =\bigg|\frac{T^3 (T_2 + t)^3 (4T_1- 3T_2+t)}{ (T_1-T_2)^4 (T_1 +t)^4}\bigg| \leq   \frac{4 T^3}{|T_1 +t|^4}.
\end{equation*}
Using this estimate in the above inequality, we obtain
\begin{align*}
&\int_{U^-} e^{2\alpha r^\ve}  r^{-2\alpha-4+\ve} u^2 \eta^2 \left( C r^3  \frac{\eta_t^2}{\eta^2} - 2\alpha^3 \right) 
 \leq  \int_{U^-} e^{2\alpha r^\ve} r^{-2\alpha-4+\ve} u^2 \eta^2 \left(C r^3  \frac{T^6}{(T_1 +t)^8}  - 2 \alpha^3 \right).
\end{align*}
Next, we write $U^- = D \cup (U^-\setminus D)$, where $D$ is the region in $U^-$ where the inequality
\begin{equation}\label{al}
2 \alpha^{3} \leq  Cr^3 \frac{T^6}{(T_1+t)^8}
\end{equation}
holds. Since we clearly have $\int_{U^-\setminus D} e^{2\alpha r^\ve} r^{-2\alpha-3} u^2 \eta^2 \left(C r^3  \frac{T^6}{(T_1 +t)^8}  - 2 \alpha^3 \right) \le 0$, we obtain
\begin{align}\label{c7}
&\int_{U^-} e^{2\alpha r^\ve}   r^{-2\alpha-4+\ve} u^2 \eta^2 \left( C r^3  \frac{\eta_t^2}{\eta^2} - 2\alpha^3 \right)
\\
& \leq   \int_{D} e^{2\alpha r^\ve}  r^{-2\alpha-4+\ve} u^2 \eta^2 \left(C r^3  \frac{T^6}{(T_1 +t)^8} - 2 \alpha^3 \right)
\notag
\\
& \leq C \int_{D} e^{2\alpha r^\ve} r^{-2\alpha-4+\ve} \eta u^2   \frac{\eta T^6}{(T_1+ t)^8}.
 \notag
\end{align}
Comparing the right-hand side of \eqref{c7} with that of \eqref{c8}, it should be clear to the reader that, in order to establish \eqref{c8}, it suffices at this point to be able to bound from above in $D$ the quantity $r^{-2\alpha-4+\ve} e^{2\alpha r^\ve} \eta \frac{\eta T^6}{(T_1+ t)^8}$. We accomplish this by first observing that, thanks to the exponential vanishing of $\eta$ at $t = - T_1$, see \eqref{d4}, we obtain for $t \in (-T_1, -T_2)$, 
\begin{equation}\label{bn1}
\frac{\eta T^6}{(T_1+ t)^8} \leq 	C,
\end{equation}
for some universal $C>0$ (depending on $T$). Secondly, we show that, thanks to the inequality \eqref{al}, the following holds in the region $D$ provided that we choose the parameter $\alpha$ large enough 
\begin{equation}\label{claim}
r^{-2\alpha - 4+\ve} e^{2\alpha r^\ve} \eta \leq 1.
\end{equation}
Using the expression \eqref{d4} for $\eta(t)$, we see that \eqref{claim} does hold in $D$ if and only if for $\alpha$ sufficiently large we have in such set
\begin{equation}\label{claima}
(2\alpha + 4-\ve) \log r + \frac{T^3(T_2+t)^4}{(T_1 +t)^3(T_1-T_2)^4} - 2\alpha r^{\ve} \ge 0.
\end{equation}
To prove \eqref{claima} observe that \eqref{al} can be equivalently written in $D$ as
\begin{equation*}
\frac{T_1 +t}{T} \leq \left(\frac{C}{2T^2}\right)^{1/8}\left(\frac{r}{\alpha}\right)^{3/8}
 = C \left(\frac{r}{\alpha}\right)^{3/8},
\end{equation*}
for some universal $C>0$. Since for $\alpha$  sufficiently large we trivially have 
\[
C \left(\frac{r}{\alpha}\right)^{3/8} \le C \left(\frac{R}{\alpha}\right)^{3/8} \le \frac1{12}, 
\]
we conclude that in $D$ we must have
\begin{equation}\label{p1}
\frac{T_1 +t }{T} \leq \frac1{12},
\end{equation}
if $\alpha>1$ has been chosen large enough. Since $\frac T4 = T_1 - T_2 = T_1 + t + |T_2 + t|$, from \eqref{p1} we conclude that we must have in $D$
\[
|T_2 + t| \geq \frac{T}{6}.
\]
If we now use this bound from below along with \eqref{al}, we find in $D$
\begin{align}\label{p2}
& (2\alpha + 4-\ve) \log r + \frac{T^3(T_2+t)^4}{(T_1 +t)^3(T_1-T_2)^4} - 2\alpha r^\ve 
\\
& \ge \left(\frac{4}6\right)^{4} \left(\frac 2C\right)^{3/8} T^{3/4} \left(\frac{\alpha}r\right)^{9/8} - (2\alpha + 4-\ve) \log\frac 1r  
\notag
\\
&- 2 \alpha r^{\ve} = C \left(\frac{\alpha}r\right)^{9/8} - (2\alpha + 4-\ve) \log\frac 1r - 2\alpha r^\ve \ge 0,
\notag
\end{align}
provided that $r<r_3\le1$, and that $\alpha$ is sufficiently large. We stress here the critical role of the power $\alpha^{9/8}$, versus the linear term $2\alpha + 4-\ve$, in reaching the above conclusion. This is precisely why we needed  to incorporate the subcritical term  $\alpha^3 \int r^{-2\alpha-4+\ve} e^{2\alpha r^\ve} u^2$ in our main Carleman estimate \eqref{est1}. We have thus proved \eqref{claima}, and consequently \eqref{claim}. Combining \eqref{c7}, \eqref{bn1}
and \eqref{claim}, we conclude that \eqref{c8} holds.

Using now the estimates \eqref{spt}, \eqref{g1}, \eqref{g2} and \eqref{c8} in \eqref{et22}, we find 
 \begin{align}\label{et02}
& \alpha \int r^{-2\alpha- 4} e^{2\alpha r^\ve} v^2   \leq C \int_{  \{r_1/2 < |x| < r_1\} \times (-T_1,T_1)} e^{2\alpha r^\ve}  ( r^{-2\alpha-4} u^2  + r^{-2\alpha-2} |\nabla u|^2)
\\
&   + C  r_2^{-2\alpha-4} e^{2\alpha r_2^\ve}  \int_{\{r_2 < |x| < r_3\} \times (-T_1,T_1)  } (u^2 +|\nabla u|^2 )\notag
\\
& + C\left( \frac{r_1}{C_2}\right)^{-2\alpha} \int_{\{r_1/2 < |x| < r_1\} \times (-T_1,T_1)}  u^2 
  + C \int_{B_R \times (-T, T)} u^2.
\notag
\end{align}
Note that in \eqref{et02} we have let $C_2= 2C_1$, with $C_1$ as in \eqref{g1}. 

Now by Lemma \ref{cac} and \eqref{monotone} it follows that for some universal $C_4$,
\[
\int_{  \{r_1/2 < |x|< r_1\} \times (-T_1,T_1)}  r^{-2\alpha - 2} e^{2\alpha r^\ve}  |\nabla u|^2 \leq C \left(\frac{r_1}{C_4}\right)^{-2\alpha - 4} \int_{ \{r_1/4 < |x|< 3r_1/2 \}  \times (-T,T)} u^2,
\]
and also
\[
\int_{  \{r_2 < |x| < r_3\} \times (-T_1,T_1)}  |\nabla u|^2 \leq C \int_{ B_R   \times (-T, T)}  u^2,
\]
where the constant in the latter estimate depends also on the parameters $r_2< r_3\le 1$, which are finally fixed at this point. Substituting these bounds in \eqref{et02}, we conclude that the following inequality holds for some new  universal constants $C$ and  $C_1$,
\begin{align}\label{et4}
& \alpha \int r^{-2\alpha- 4} e^{2\alpha r^\ve} v^2  \leq C \left(\frac{r_1}{C_1}\right)^{-2\alpha - 4}  \int_{  \{r_1/4 < |x| < 3r_1/2\} \times (-T, T)}  u^2  
\\
&   + C r_2^{-2\alpha-4} e^{2\alpha r_2^\ve} \int_{B_R \times (-T,T)} u^2. \notag
\end{align}
The integral in the left-hand side of \eqref{et4} can be bounded from below in the following way using \eqref{monotone},
\begin{equation}\label{b5}
\alpha \int r^{-2\alpha- 4} e^{2\alpha r^\ve} v^2 \geq   \alpha r_2^{-2\alpha - 4}  e^{2\alpha r_2^\ve} \int_{\{r_1 < |x|< r_2  \} \times (-T_2,T_2)} u^2. 
\end{equation}
Substituting \eqref{b5} in \eqref{et4}, and dividing both sides by $r_2^{-2\alpha - 4} e^{2\alpha r_2^\ve}$, we obtain
\[
\alpha \int_{\{r_1 < |x| < r_2\} \times (-T_2,T_2)} u^2 \leq C \left(\frac{r_1}{C_1r_2}\right)^{-2\alpha -4}  \int_{  \{r_1/4 < |x| < 3r_1/2\} \times (-T, T)  }  u^2 + C  \int_{  B_R \times (-T, T)  } u^2.
\]
Now adding $\alpha \int_{ B_{r_1} \times (-T_2,T_2)} u^2$ to both sides of the latter inequality, we  obtain 
\begin{align}\label{et5}
& \alpha \int_{B_{r_2} \times (-T_2, T_2)} u^2  \leq   C \left(\frac{r_1}{C_1r_2}\right)^{-2\alpha -4}  \int_{  \{r_1/4 < |x| < 3r_1/2\} \times (-T, T)}  u^2
\\
& +  \alpha \int_{ B_{r_1} \times (-T_2, T_2)} u^2+ C   \int_{  B_R \times (-T, T)  } u^2
\notag
\\
& \le 2C \left(\frac{r_1}{C_1r_2}\right)^{-2\alpha - 4}  \int_{  B_{3r_1/2}  \times (-T, T)  }  u^2
 +C  \int_{  B_R \times (-T, T)  } u^2, \notag\end{align}
where, recalling our initial choice $r_1 < r_2$, we note that in the second inequality in \eqref{et5} we have used  
\[
 \alpha \leq   \left(\frac{r_1}{C_2r_2}\right)^{-2\alpha - 4}.
 \]
Keeping in mind the hypothesis \eqref{assume}, we now choose $\alpha$ (depending on $u$) such that
\begin{equation}\label{choice2}
\alpha \int_{ B_{r_2} \times (-T_2,T_2)} u^2  \geq  2   C  \int_{  B_R \times (-T, T)  } u^2, 
\end{equation}
where $C$ is as in \eqref{et5}. Thus, by subtracting off $ C \int_{B_R \times (-T, T)} u^2$ from both sides of \eqref{et5},  we obtain 
 \begin{align}\label{et7}
\left(\frac{r_1}{C_1 r_2}\right)^{2 \alpha + 4}  \frac{\alpha}{2}  \int_{B_{r_2} \times (-T_2, T_2)} u^2 \leq 2 C \int_{ B_{3r_1/2} \times (-T, T) } u^2.
\end{align}
At this point, we fix $\alpha$ sufficiently large in such a way that \eqref{choice1}, \eqref{p1}, \eqref{p2} and \eqref{choice2} simultaneously hold. Letting $3r_1/2= s$, we obtain from \eqref{et7} that for some new constants $C, A$ depending on  $r_2, r_3, R$, the ratio $\frac{\int_{  B_R \times (-T, T)  } u^2}{\int_{ B_{r_2} \times (-T_2,T_2)} u^2}$,  and $\alpha$ (which at this point is fixed),  the following holds for all $0< s \leq r_2/8$,
\[
\int_{ B_{s} \times (-T, T) } u^2 \geq C s^{A}.
\]
Since this estimate is in contradiction with the hypothesis that $u$ parabolically vanish to infinite order in the sense of  \eqref{vp}, we have finally proved the theorem.

   \end{proof}

\end{document}